\documentclass[12pt,a4paper]{amsart}
\usepackage{amsfonts,color}
\usepackage{amsthm}
\usepackage{amsmath}
\usepackage{amscd}
\usepackage[latin2]{inputenc}
\usepackage{t1enc}
\usepackage[mathscr]{eucal}
\usepackage{indentfirst}
\usepackage{graphicx}
\usepackage{graphics}
\usepackage{pict2e}
\usepackage{epic}
\usepackage{url}
\usepackage{epstopdf}

\numberwithin{equation}{section}
\usepackage[margin=2.6cm]{geometry}

 \usepackage{tikz}
\usepackage{tkz-graph}
\usetikzlibrary{arrows,shapes,snakes,automata,backgrounds,petri}
\tikzstyle{vertex}=[draw=black,circle,fill=black,minimum size=4pt, inner sep=0pt, outer sep=0pt,text=white,line width=0mm]
\tikzstyle{c0}=[shape=circle, minimum size=4pt, fill=white]
\tikzstyle{c1}=[shape=rectangle, minimum size=7pt, fill=red]
\tikzstyle{c2}=[shape=diamond, minimum size=10pt, fill=blue]

\theoremstyle{plain}
\newtheorem{Th}{Theorem}[section]
\newtheorem{Lemma}[Th]{Lemma}

\newtheorem{Prop}[Th]{Proposition}

 \theoremstyle{definition}

\newtheorem{Rem}[Th]{Remark}
\newtheorem{?}[Th]{Problem}

\newcommand{\C}{\mathbb{C}}

\begin{document}

\title{Note on the zero-free region of the hard-core model}

\author[F. Bencs]{Ferenc Bencs}

\address{
Central European University, Department of Mathematics, H-1051 Budapest,
Zrinyi u. 14., Hungary \& Alfr\'ed R\'enyi Institute of Mathematics, H-1053 Budapest,
Re\'altanoda u. 13-15}

\email{ferenc.bencs@gmail.com}

\author[P. Csikv\'ari]{P\'{e}ter Csikv\'{a}ri}

\address{MTA-ELTE Geometric and Algebraic Combinatorics Research Group \\ P\'{a}zm\'{a}ny P\'{e}ter s\'{e}t\'{a}ny 1/C \\ Hungary \& E\"{o}tv\"{o}s Lor\'{a}nd University \\ Mathematics Institute, Department of Computer 
Science \\ H-1117 Budapest
\\ P\'{a}zm\'{a}ny P\'{e}ter s\'{e}t\'{a}ny 1/C \\ Hungary} 

\email{peter.csikvari@gmail.com}

\thanks{The first author is partially supported by the MTA R\'enyi Institute Lend\"ulet Limits of Structures Research Group. The second author  is supported by the 
 Marie Sk\l{}odowska-Curie Individual Fellowship grant no. 747430, and before that grant he was partially supported by the
Hungarian National Research, Development and Innovation Office, NKFIH grant K109684 and Slovenian-Hungarian grant NN114614, and by the ERC Consolidator Grant  648017.}



\begin{abstract} In this paper we prove a new zero-free region for the partition function of the hard-core  model, that is, the independence polynomials of graphs with largest degree $\Delta$. This new domain contains the half disk 
$$D=\left\{ \lambda \in \mathbb{C}\ |\ \mathrm{Re}(\lambda)\geq 0, |\lambda|\leq \frac{7}{8}\tan \left( \frac{\pi}{2(\Delta-1)}\right)\right\}.$$
\end{abstract}

\maketitle

\section{Introduction}

The independence polynomial of a graph $G$ is defined as 
$$Z_G(\lambda)=\sum_{k=0}i_k(G)\lambda^k,$$
where $i_k(G)$ be the number of independent sets of the graph $G$. More generally one can define a multivariate version of the independence polynomial as follows. For each vertex $v$ let us introduce a new variable $\lambda_v$, and set
$$Z_G(\{\lambda_v\})=\sum_{I\in \mathcal{I}(G)}\prod_{v\in I}\lambda_v,$$
where $\mathcal{I}(G)$ is the set of independent sets of $G$ including the empty one.
The polynomial $Z_G(\{\lambda_v\})$ is the partition function of the hard-core model in statistical physics. It is an important problem to study its zero-free region. 
It is related to many things including Lov\'asz local lemma, the computational complexity of the computation of independence polynomial. One of the best-known result is due to Shearer showing that $Z_G(\lambda)$ does not vanish in a certain disk.

\begin{Th}[J. Shearer \cite{shearer1985problem}] \label{Shearer} Suppose that a graph $G$ has largest degree at most $\Delta$, and $|\lambda_v|\leq \frac{(\Delta-1)^{\Delta-1}}{\Delta^{\Delta}}$ for each vertex $v\in V(G)$. Then $Z_G(\{\lambda_v\})\neq 0$. In particular, $Z_G(\lambda)\neq 0$ if $|\lambda|\leq \frac{(\Delta-1)^{\Delta-1}}{\Delta^{\Delta}}$.
\end{Th} 

\noindent Very recently H. Peters and G. Regts  found other zero-free regions.

\begin{Th}[H. Peters and G. Regts \cite{peters2017conjecture}] \label{PR1} There exists an open domain $D$ on the complex plane that contains the interval $(0, \frac{(\Delta-1)^{\Delta-1}}{(\Delta-2)^{\Delta}})$ such that $Z_G(\lambda)\neq 0$ if $\lambda\in D$ and $G$ has largest degree at most $\Delta$.
\end{Th}

\noindent They also gave a more explicit zero-free region.

\begin{Th}[H. Peters and G. Regts \cite{peters2017conjecture}] \label{PR2}
Let $G$ be a graph with largest degree $\Delta\geq 2$. For a fixed $\varepsilon$ et
$$D_{\varepsilon}=\left\{ \lambda\in \mathbb{C}\ |\ |\lambda|\leq \tan\left(\frac{\pi}{(2+\varepsilon)(\Delta-1)}\right)\ \ \mbox{and}\ \ |\mathrm{arg}(\lambda)|\leq \frac{\varepsilon \pi}{2(2+\varepsilon)}\right\}.$$
Set $D_{PR}=\bigcup_{\varepsilon}D_{\varepsilon}$. Suppose that for each vertex $v\in V(G)$ we have $\lambda_v\in D_{\varepsilon}$ for a fixed $\varepsilon$, then
$Z_G(\{\lambda_v\})\neq 0$. In particular, if $\lambda \in D_{PR}$ then $Z_G(\lambda)\neq 0$.
\end{Th}

\noindent At the same time H. Peters and G. Regts \cite{peters2017conjecture} also identified a domain $U_{d}$ for which the zeros of independence polynomials of graphs of maximum degree $\Delta=d+1$ is dense on the complement of $U_{d}$. This domain can be described as follows. Let $d=\Delta-1$, and let
$$U_d=\left\{ \frac{-\alpha d^d}{(d+\alpha)^{d+1}}\ \Bigg| |\alpha|\leq 1\right\}.$$
In fact, H. Peters and G. Regts showed the following stronger result.  Consider the  complete $d$-ary trees $T_{k,d}$, where all leaves have distance $k$ from the root. Then the zeros of independence polynomials of the trees $T_{k,d}$ are dense outside $U_d$. 

\noindent Below we describe our contribution. Instead of giving immediately the explicit description of the our zero-free domain $D$ we only give some of its most important properties.

\begin{Th} \label{circular}
(a) Let $D=\{\lambda\ |\ \mathrm{Re}\lambda\geq 0, |\lambda|\leq 1\}$. If $\lambda\in D$ then $Z_G(\lambda)\neq 0$ for any graph $G$ with maximum degree $3$.
\medskip

(b) Let $\Delta \geq 3$. Let 
$$D_1=\left\{\lambda\ |\ \mathrm{Re}\lambda\geq 0, |\lambda|\leq \frac{7}{8}\tan \left(\frac{\pi}{2(\Delta-1)}\right) \right\}$$
and  
$$D_2=\left\{\lambda\ |\ \mathrm{Re}\lambda\geq 0, |\lambda|\leq \tan \left(\frac{\pi}{2(\Delta-1)}\right) \right\}.$$
There exists an open domain $D\subseteq \C$ with the following properties:
\begin{enumerate}
\item if $\lambda\in D$ then $Z_G(\lambda)\neq 0$ for any graph $G$ with maximum degree $\Delta$,
\item $D_1\subseteq D\subseteq D_2$,
\item $D_{PR}\subset D$,
\item $\tan \left(\frac{\pi}{2(\Delta-1)}\right),\pm i\tan \left(\frac{\pi}{2(\Delta-1)}\right)\in D$.
\end{enumerate}
\end{Th}

\begin{figure}[h]
   \centering
   \includegraphics[width=8cm]{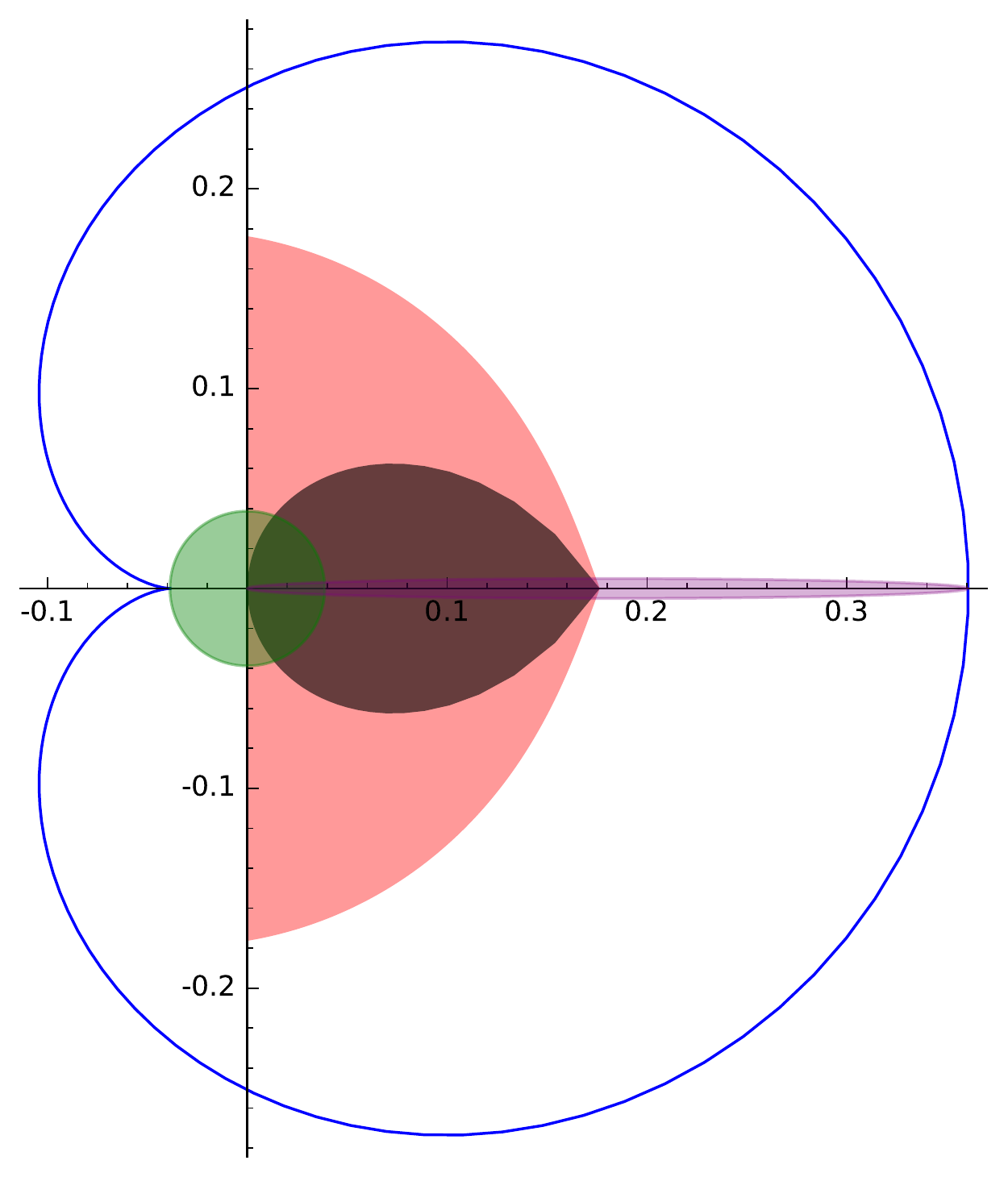}
   \caption{Zero-free regions for the independence polynomial of graphs of maximal degree at most $\Delta=10$. The green domain is Shearer's disk, the brown berry-shape domain is $D_{PR}$ and the red one is our new domain. The blue curve is the boundary of $U_9$.}
   \label{fig:reg}
\end{figure}

The zeros of the independence polynomial are strongly connected to the computational complexity of approximating the independence polynomial. I. Bez\'akov\'a, A. Galanis, L. A. Goldberg and \v{S}tefankovi\v{c} \cite{bezakova2018inapproximability} showed that 
outside $U_d$ approximating $Z_G(\lambda)$ on graphs $G$ with maximum degree at most $\Delta$ is NP-hard. In fact, they also showed that if $\lambda$ is not a positive real number outside $U_d$, then approximating the independence polynomial is $\#$P-hard.
For negative $\lambda$ outside $U_d$ they also proved that it is  $\#$P-hard to decide whether $Z_G(\lambda)>0$. Inside $U_d$ the picture is much more pleasing. First, D. Weitz showed that there is
a (deterministic) fully polynomial time approximation algorithm (FPTAS) for computing
$Z_G(\lambda)$ for any $0\leq \lambda < \frac{(\Delta-1)^{\Delta-1}}{(\Delta-2)^{\Delta}}$ for any graph of maximum degree at most $\Delta$. V. Patel and G. Regts \cite{patel2017deterministic} showed that if $Z_G(\lambda)$ does not vanish in a domain then there is a polynomial time algorithm for approximating the partition function. Their method is based on the interpolation method introduced by A. Barvinok \cite{barvinok2016combinatorics}.
\bigskip

\section{Preliminaries}

Our proof of Theorem~\ref{circular} is inspired by the proof of Theorem~\ref{PR2}, but we simplify some arguments and add some extra ideas. One of the simplifications is based on the observation that it is enough to prove Theorem~\ref{circular} for trees as the following lemma shows.

\begin{Lemma}[F. Bencs \cite{bencs2017trees}] \label{tree} For every graph $G$ with maximum degree at most $\Delta$ there exists a tree $T_G$ with maximum degree at most $\Delta$ such that $Z_G(\lambda)$ divides $Z_{T_G}(\lambda)$ as a polynomial.
\end{Lemma}

There are many choices for the tree $T_G$ in Lemma~\ref{tree}. This  lemma is implicit in the work of D. Weitz
\cite{weitz2006counting} and in the work of A. D. Scott and A. D. Sokal \cite{scott2005repulsive}, but the above algebraic formulation is due to F. Bencs \cite{bencs2017trees}.
\bigskip

Given a tree $T$ with maximum degree $\Delta\geq 2$ let us pick one of its leaf as a root vertex $v$. Then every vertex has at most $\Delta-1$ children. If $T'$ is a subtree of $T$ then we can choose a root vertex of $T'$ in a natural way: the vertex $v'$ closest to $v$. Following H. Peters and G. Regts \cite{peters2017conjecture} for a fixed $\lambda$ let us introduce the quantity
$$R_{G,v}=\frac{\lambda Z_{G\setminus N[v]}(\lambda)}{Z_{T-v}(\lambda)}.$$
Note that for an arbitrary graph $G$ and vertex $v$ we have 
$$Z_G(\lambda)=Z_{G-v}(\lambda)+\lambda Z_{G-N[v]}(\lambda),$$
where $N[v]=N(v)\cup \{v\}$, the closed neighborhood of $v$.  Hence $Z_G(\lambda)\neq 0$ if and only if $R_{G,v}\neq -1$. Now if $T'$ is a tree with root vertex $v'$, and $u_1,\dots ,u_k$ are the neighbours of $v'$ then a simple computation shows that 
$$R_{T',v'}=\frac{\lambda}{\prod_{i=1}(1+R_{T_i,u_i})},$$
where $T_i$ is the subtree of $T'-v'$ rooted at $u_i$. Here $k\leq \Delta-1=d$. For a fixed $\lambda$ let us consider the map
$$f_{\lambda}(z_1,\dots ,z_d)=\frac{\lambda}{\prod_{i=1}^d(1+z_i)}.$$
It is not a problem that in the above recusrion $k$ may be smaller than $d$ since we can substitute $z_i=0$ for those variables that we do not need.

Now let $S_{\lambda}$ be the set of elements of the complex plane that can be obtained by the rules (1) $0\in S_{\lambda}$, (2) $z_1,\dots ,z_d\in S_{\lambda}$ then $f_{\lambda}(z_1,\dots ,z_d)\in S_{\lambda}$. The following lemma is now trivial.

\begin{Lemma} \label{general}
For a fixed $\lambda$ let us consider the map
$$f(z_1,\dots ,z_d)=\frac{\lambda}{\prod_{i=1}^d(1+z_i)}.$$
If there is an open domain $F_{\lambda}$ such that $0\in F_{\lambda},-1\notin \overline{F_{\lambda}}$ and for $z_1,\dots ,z_d \in F_{\lambda}$ we have $f(z_1,\dots ,z_d)\in F_{\lambda}$ then $Z_G(\lambda)\neq 0$ for any graph $G$ with maximum degree $d+1$.
\end{Lemma}

\noindent We collected some results on $S_{\lambda}$.

\begin{Prop}[H. Peters, G. Regts \cite{peters2017conjecture}] If $\lambda \notin U_d$ then $-1\in \overline{S_{\lambda}}$.
\end{Prop}

\begin{Prop} If $-1\notin \overline{S_{\lambda}}$, and $z\in S_{\lambda}$ then $\frac{\lambda}{1+z}\in U_{d-1}$.
\end{Prop}

\begin{proof} For $z\in S_{\lambda}$ simply consider the map $f_{\lambda}(z_1,\dots ,z_{d-1},z)=f_{\lambda/(1+z)}(z_1,\dots ,z_{d-1})$ and apply the previous proposition.
\end{proof}

\begin{Rem} Suppose that in a small neighborhood of $\lambda$, the polynomials of $Z_G$ do not vanish for each graph $G$ of largest degree at most $\Delta$. Equivalently, $-1\notin \overline{S_{\lambda}}$. Suppose that $|\lambda|\geq \frac{d^d}{(d+1)^{d+1}}$, that is, $\lambda$ is outside of Shearer's disk.
By the above proposition we get that if $z\in S_{\lambda}$, then $\frac{\lambda}{1+z}\in U_{d-1}$. Since $U_{d-1}$ is contained in a disk of radius $\frac{(d-1)^{d-1}}{(d-2)^d}$ we get that
$$\bigg|\frac{\lambda}{1+z}\bigg|\leq \frac{(d-1)^{d-1}}{(d-2)^d}$$
whence
$$|1+z|\geq |\lambda|\frac{(d-2)^d}{(d-1)^{d-1}}\geq \frac{(d-2)^dd^d}{(d-1)^{d-1}(d+1)^{d+1}}.$$
So there is an explicit ball around $-1$ that is avoided by $S_{\lambda}$.

\end{Rem}

\begin{Rem} The advantage of Lemma~\ref{tree} is that we can avoid the two-round argument that is typical in the field. In this two-round argument first one proves an auxiliary claim for pairs $(G,v)$, where all neighbors of $v$ has degree at most $\Delta-1$, and in the second round one can prove the claim in the general case. For a prototype of such an argument see the proof of Theorem 2.3 of  \cite{peters2017conjecture}.
\end{Rem}

\section{Proof of Theorem~\ref{circular}}

\noindent In this section we prove Theorem~\ref{circular}.
Recall that $d=\Delta-1$. First we specify Lemma~\ref{general} to a special domain $F_{\lambda}$.

\begin{Lemma} \label{angle1}
Let $\lambda$ be fixed, and $\arg \lambda=\alpha \in (-\pi/2,\pi/2)$ . Let $0\leq \beta,\gamma \leq \pi/2$. Set
$$\beta'=\arctan \frac{|\lambda|\sin \beta}{1+|\lambda|\cos \beta}\ \ \ \mbox{and}\ \ \ \gamma'=\arctan \frac{|\lambda|\sin \gamma}{1+|\lambda|\cos \gamma}.$$
Suppose that $d\gamma'-\beta \leq \alpha \leq \gamma-d\beta'$. Then $Z_G(\lambda)\neq 0$ for any graph $G$ with maximum degree $d+1$.
\end{Lemma}

\begin{proof}
We will show that the domain
$$F_{\lambda}=\{ z\ |\ |z|\leq |\lambda|, -\beta\leq \arg z\leq \gamma\}$$
satisfies the requirements above. Then Lemma~\ref{general} implies the claim. It is clear that $0\in F_{\lambda}$ and $-1\notin \overline{F_{\lambda}}$. Now suppose that $z_1,\dots ,z_d \in F_{\lambda}$. Then
$$|f(z_1,\dots ,z_d)|=\frac{|\lambda|}{\prod_{i=1}^d|1+z_i|}\leq |\lambda|$$
since $|1+z_i|\geq 1$ if $z_i\in F_{\lambda}$. Furthermore,
$$\arg f(z_1,\dots ,z_d)=\arg \lambda-\sum_{i=1}^d\arg (1+z_i).$$
Note that for $z_i\in F_{\lambda}$ we have $-\beta'\leq \arg(1+z_i)\leq \gamma'$ since we have $\{1+z\ |\ z\in F_{\lambda}\}\subseteq \{z'\ |\ -\beta'\leq \arg z'\leq \gamma'\}$. 
Hence
$$-\beta\leq \alpha-d\gamma'\leq \arg f(z_1,\dots ,z_d)\leq \alpha+d\beta'\leq \gamma$$
implying that $f(z_1,\dots ,z_d)\in F_{\lambda}$.
\end{proof}

Now we reverse our argument. We first fix $\beta$ and $\gamma$ and we choose the largest possible $|\lambda|$ and the corresponding $\alpha$. The only condition is that
$d\gamma'-\beta \leq \gamma-d\beta'$, that is, $d(\beta'+\gamma')\leq \beta+\gamma$. As we increase $|\lambda|$ the angles $\beta',\gamma'$ will increase too. So for the largest possible $|\lambda|$ we have $d(\beta'+\gamma')= \beta+\gamma$. We need that
$$d\left(\arctan \frac{|\lambda|\sin \beta}{1+|\lambda|\cos \beta}+\arctan \frac{|\lambda|\sin \gamma}{1+|\lambda|\cos \gamma}\right)=\beta+\gamma.$$
Since
$$\arctan x+\arctan y=\arctan \left(\frac{x+y}{1-xy}\right),$$
and $\tan$ is a monotone increasing function on $(0,\pi/2)$ we get that
$$\frac{\frac{|\lambda|\sin \beta}{1+|\lambda|\cos \beta}+\frac{|\lambda|\sin \gamma}{1+|\lambda|\cos \gamma}}{1-\frac{|\lambda|\sin \beta}{1+|\lambda|\cos \beta}\cdot \frac{|\lambda|\sin \gamma}{1+|\lambda|\cos \gamma}}=\tan \left(\frac{\beta+\gamma}{d}\right).$$
This simplifies to
$$\frac{|\lambda|(\sin \beta+\sin \gamma)+|\lambda|^2\sin(\beta+\gamma)}{1+|\lambda|(\cos \beta+\cos \gamma)+|\lambda|^2\cos(\beta+\gamma)}=\tan \left(\frac{\beta+\gamma}{d}\right).$$
After multiplying with the denominator on the left hand side we get a quadratic equation. Let
$$A=\sin(\beta+\gamma)-\tan\left(\frac{\beta+\gamma}{d}\right)\cos(\beta+\gamma),$$
$$B=\sin \beta+\sin \gamma-\tan\left(\frac{\beta+\gamma}{d}\right)(\cos \beta+\cos \gamma),$$
$$C=-\tan\left(\frac{\beta+\gamma}{d}\right),$$
and $0=A|\lambda|^2+B|\lambda|+C$. Note that
$$\frac{\sin \beta+\sin \gamma}{\cos \beta+\cos \gamma}=\tan \left(\frac{\beta+\gamma}{2}\right).$$
From this it follows that $A>0, B\geq 0, C<0$ since $0\leq \beta+\gamma\leq \pi$, and $d\geq 2$. Hence exactly one of the solutions of the quadratic equation $Ax^2+Bx+C=0$ is positive. Set
$$s(\beta,\gamma)=\frac{-B+\sqrt{B^2-4AC}}{2A}.$$
For $d=2$ it turns out that $s(\beta,\gamma)\equiv 1$. For $d\geq 2$ it is always true that
$$s(\beta,\pi/2)>\frac{2\beta+\pi}{\pi(1+\sin \beta)}\tan \left(\frac{\pi}{2d}\right)>\frac{7}{8}\tan \left(\frac{\pi}{2d}\right).$$
(Some comment about these facts. The function $g(d)=\frac{s_d(\beta,\pi/2)}{\tan \left(\frac{\pi}{2d}\right)}$ is monotone decreasing for fixed $\beta$, and its limit is $\frac{2\beta+\pi}{\pi(1+\sin \beta)}$. The justifications of these facts are standard but somewhat tedious computations.)
It turns out that $s(\beta,\pi/2)$ is a convex function on $[0,\pi/2]$ and at the end points, $s(0,\pi/2)=s(\pi/2,\pi/2)=\tan\left(\frac{\pi}{2d}\right)$.
\medskip

Now let us turn our attention to the angle $\alpha=t(\beta,\gamma)=\gamma-d\beta'$, where the length $|\lambda|$ is fixed to be $s(\beta,\gamma)$. Here we only study the angle $t(\beta,\pi/2)$. Since $s(\beta,\pi/2)$ is continuous we get that $\beta'$ is continuous, and so $t(\beta,\pi/2)$ is continuous too. For $\beta=0$ we get $\beta'=0$ independently of $|\lambda|$, and so $t(0,\pi/2)=\pi/2$. For $\beta=\pi/2$ the equation $d(\beta'+\gamma')=\beta+\gamma$ implies that $t(\pi/2,\pi/2)=0$. Hence $t(\beta,\pi/2)$ takes all values in the interval $[0,\pi/2]$. 

Now we are ready to prove Theorem~\ref{circular}.

\begin{proof}[Proof of Theorem~\ref{circular}]
Consider the domain
$$D=\left\{\lambda \in \C\ |\ \exists \beta \in [0,\pi/2]: |\arg \lambda|=t(\beta,\pi/2), |\lambda|\leq s(\beta,\pi/2)\right\}.$$
Then by the bounds on $s(\beta,\pi/2)$ we have $D_1\subseteq D\subseteq D_2$, and 
$\tan \left(\frac{\pi}{2(\Delta-1)}\right),\pm i\tan \left(\frac{\pi}{2(\Delta-1)}\right)\in D$. Furthermore, for $\Delta=3$ we have $s(\beta,\pi/2)\equiv 1$.

Now suppose that $\lambda\in D$, we can assume that $\arg \lambda=\alpha\geq 0$. Then there exists a $\beta\in [0,\pi/2]$ such that $\arg \lambda=t(\beta,\pi/2)$ and $|\lambda|\leq s(\beta,\pi/2)$. Set $\gamma=\pi/2$ and
$$\beta_1'=\arctan \frac{|\lambda|\sin \beta}{1+|\lambda|\cos \beta}\ \ \ \mbox{and}\ \ \ \gamma_1'=\arctan \frac{|\lambda|\sin \gamma}{1+|\lambda|\cos \gamma},$$
and
$$\beta_2'=\arctan \frac{s(\beta,\pi/2)\sin \beta}{1+s(\beta,\pi/2)\cos \beta}\ \ \ \mbox{and}\ \ \ \gamma_2'=\arctan \frac{s(\beta,\pi/2)\sin \gamma}{1+s(\beta,\pi/2)\cos \gamma}.$$
Since $|\lambda|\leq s(\beta,\pi/2)$ we get that  $\beta_1'\leq \beta_2'$ and $\gamma_1'\leq \gamma_2'$. Furthermore,
$$\alpha=\gamma-d\beta_2'=d\gamma_2'-\beta$$
by the definition of $s(\beta,\gamma)$. Hence
$$d\gamma_1'-\beta\leq \alpha \leq \gamma-d\beta_1'.$$
By Lemma~\ref{angle1} we have $Z_G(\lambda)\neq 0$ for any graph $G$ with maximum degree $\Delta$.

The fact that $D_{PR}$ is contained in $D$ follows from the fact that their argument of zero-freeness is a variation of the special case of Lemma~\ref{angle1} with $\beta=0,\gamma=\pi/2$.

\end{proof}

\section{Analysis}

Here we collected the proofs of some inequalities that we used in the proof of Theorem~\ref{circular}.

\begin{Lemma} Let $x_0$ be the positive zero of the equation $Ax^2+Bx+C=0$, where $A,B>0$ and $C<0$. Then 
$$-\frac{C}{B}-\frac{AC^2}{B^3}\leq x_0\leq -\frac{C}{B}.$$
\end{Lemma}

\begin{proof}
Since $0=Ax_0^2+Bx_0+C\geq Bx_0+C$ we immediately get that $x_0\leq -\frac{C}{B}$. Then
$$0=Ax_0^2+Bx_0+C\geq Bx_0\leq A\left(-\frac{C}{B}\right)^2+Bx_0+C.$$
Hence $x_0\geq -\frac{C}{B}-\frac{AC^2}{B^3}$.
\end{proof}

Recall that
$$A=\sin(\beta+\gamma)-\tan\left(\frac{\beta+\gamma}{d}\right)\cos(\beta+\gamma),$$
$$B=\sin \beta+\sin \gamma-\tan\left(\frac{\beta+\gamma}{d}\right)(\cos \beta+\cos \gamma),$$
$$C=-\tan\left(\frac{\beta+\gamma}{d}\right),$$
where we have chosen $\gamma=\frac{\pi}{2}$. By introducing the notation $U=\tan\left(\frac{\beta+\gamma}{d}\right)$ we have
$$A=\cos(\beta)+U\sin(\beta),\ \ B=1+\sin(\beta)-U\cos(\beta),\ \ C=-U.$$
We first show a computation that motivates the more technical follow-up computation:
$$-\frac{C}{B}=\frac{U}{1+\sin(\beta)-U\cos(\beta)}\geq \frac{U}{1+\sin(\beta)}=\frac{\tan\left(\frac{\beta+\pi/2}{d}\right)}{1+\sin(\beta)}\geq \frac{2\beta+\pi}{\pi(1+\sin(\beta))}\tan \left(\frac{\pi}{2d}\right),$$
where in the last step we used the fact that the function $h(x)=\frac{\tan(x)}{x}$ is a monotone increasing function, hence $h\left(\frac{\beta+\pi/2}{d}\right)\geq h\left(\frac{\pi}{2d}\right)$. Unfortunately, in our case we need a lower bound for $x_0=s(\beta,\pi/2)$ and $-\frac{C}{B}$ is an upper bound to $x_0$. On the other hand, the inequality $-\frac{C}{B}-\frac{AC^2}{B^3}\leq x_0\leq -\frac{C}{B}$ suggest that we need to prove that $x_0\geq \frac{2\beta+\pi}{\pi(1+\sin(\beta))}\tan \left(\frac{\pi}{2d}\right)$. It also motivates the introduction of the notation
$$V=\frac{2\beta+\pi}{\pi}\tan \left(\frac{\pi}{2d}\right).$$


To show that $x_0\geq \frac{2\beta+\pi}{\pi(1+\sin(\beta))}\tan \left(\frac{\pi}{2d}\right)$ it is enough to prove that $Ax^2+Bx+C<0$ for $x=\frac{2\beta+\pi}{\pi(1+\sin(\beta))}\tan \left(\frac{\pi}{2d}\right)$. Thus we need to prove that the function
$$F(d,\beta):=(\cos(\beta)+U\sin(\beta))\left(\frac{V}{1+\sin(\beta)}\right)^2+(1+\sin(\beta)-U\cos(\beta))\left(\frac{V}{1+\sin(\beta)}\right)-U$$
is negative for all $d\geq 3$ and $\beta\in (0,\frac{\beta}{2})$. After algebraic manipulation we get that
$$F(d,\beta)=(V-U)-\frac{\cos(\beta)}{(1+\sin(\beta))^2}V(U-V)-\frac{\sin(\beta)}{(1+\sin(\beta))^2}UV(\cos(\beta)-V).$$
Note that $U\geq V$ since $h(x)=\frac{\tan(x)}{x}$ is a monotone increasing function.  Hence it is enough to show that
$$U-V\geq \frac{\sin(\beta)}{(1+\sin(\beta))^2}UV(V-\cos(\beta)).$$
This is clearly true for fixed $\beta$ and large $d$ since the right hand side is negative. Let us introduce the function
$$G(d,\beta):=U-V-\frac{\sin(\beta)}{(1+\sin(\beta))^2}UV(V-\cos(\beta)).$$
A little computation shows that $G(d,0)=G(d,\pi/2)=0$. It seems to be concave...

\bibliographystyle{siamnodash}

\bibliography{biblio_zero-free}

\begin{thebibliography}{1}

\bibitem{barvinok2016combinatorics}
{\sc A.~Barvinok}, {\em Combinatorics and complexity of partition functions},
  vol.~276, Springer, 2016.

\bibitem{bencs2017trees}
{\sc F.~Bencs}, {\em {On trees with real rooted independence polynomial}},
  arXiv preprint arXiv:1703.05409,  (2017).

\bibitem{bezakova2018inapproximability}
{\sc I.~Bez{\'a}kov{\'a}, A.~Galanis, L.~A. Goldberg, and
  D.~{\v{S}}tefankovi{\v{c}}}, {\em {Inapproximability of the independent set
  polynomial in the complex plane}}, in Proceedings of the 50th Annual ACM
  SIGACT Symposium on Theory of Computing, ACM, 2018, pp.~1234--1240.

\bibitem{patel2017deterministic}
{\sc V.~Patel and G.~Regts}, {\em {Deterministic polynomial-time approximation
  algorithms for partition functions and graph polynomials}}, SIAM Journal on
  Computing, {\bf 46} (2017), pp.~1893--1919.

\bibitem{peters2017conjecture}
{\sc H.~Peters and G.~Regts}, {\em {On a conjecture of Sokal concerning roots
  of the independence polynomial}}, arXiv preprint arXiv:1701.08049,  (2017).

\bibitem{scott2005repulsive}
{\sc A.~D. Scott and A.~D. Sokal}, {\em {The repulsive lattice gas, the
  independent-set polynomial, and the Lov{\'a}sz local lemma}}, Journal of
  Statistical Physics, {\bf 118} (2005), pp.~1151--1261.

\bibitem{shearer1985problem}
{\sc J.~B. Shearer}, {\em {On a problem of Spencer}}, Combinatorica, {\bf 5}
  (1985), pp.~241--245.

\bibitem{sly2014}
{\sc A.~Sly and N.~Sun}, {\em Counting in two-spin models on d -regular
  graphs}, Ann. Probab., {\bf 42} (2014), pp.~2383--2416.

\bibitem{weitz2006counting}
{\sc D.~Weitz}, {\em {Counting independent sets up to the tree threshold}}, in
  Proceedings of the thirty-eighth annual ACM symposium on Theory of computing,
  ACM, 2006, pp.~140--149.

\end{thebibliography}

\end{document}